\documentclass{amsart}

\title{$\mathbb{Q}$-curves and the Lebesgue--Nagell equation}
\author[Michael Bennett]{Michael A. Bennett}
\address{Department of Mathematics, University of British Columbia, Vancouver, B.C., V6T 1Z2 Canada}
\email{bennett@math.ubc.ca}

\author{Philippe Michaud-Jacobs}
\address{Mathematics Institute, University of Warwick, Coventry CV4 7AL, United Kingdom}
\email{P.Rodgers@warwick.ac.uk}

\author{Samir Siksek}
\address{Mathematics Institute, University of Warwick, Coventry CV4 7AL, United Kingdom}
\email{S.Siksek@warwick.ac.uk}
\thanks{The first-named author is supported by NSERC. The second-named author is supported by an EPSRC studentship and has previously used the name Philippe Michaud-Rodgers. The third-named author is
supported by an EPSRC Grant EP/S031537/1 \lq\lq Moduli of
elliptic curves and classical Diophantine problems\rq\rq.}

\makeatletter
\newcommand\notsotiny{\@setfontsize\notsotiny\@vipt\@viipt}
\makeatother
\usepackage{amsmath} 
\usepackage{amssymb} 
\usepackage{mathtools}
\usepackage{framed}
\usepackage{pifont} 
\usepackage{enumerate} 
\usepackage{wasysym} 
\usepackage{commath}
\usepackage{graphicx}
\usepackage{framed}
\usepackage{amsmath}
\usepackage{amssymb}
\usepackage{amsfonts}
\usepackage{mathrsfs}
\usepackage{mathtools}
\usepackage{url}
\usepackage{array}
\usepackage{amsthm}
\usepackage[english]{babel}
\usepackage[utf8]{inputenc}
\newtheorem{theorem}{Theorem}[section]
\newtheorem*{theorem*}{Theorem}

\newtheorem {corollary}[theorem]{Corollary}
\newtheorem {lemma}[theorem]{Lemma}
\newtheorem {proposition}[theorem]{Proposition}
\theoremstyle{definition}

\theoremstyle{remark}
\newtheorem{remark}[theorem]{Remark}
\newtheorem*{remark*}{Remark}
\usepackage{yfonts}
\usepackage{amsmath,amssymb,tikz-cd}
\usepackage{pdflscape}
\usepackage{makecell}

\usepackage{etoolbox}
\apptocmd{\sloppy}{\hbadness 10000\relax}{}{}

\newcommand{\ord}{\operatorname{ord}}

\usepackage[numbers]{natbib}
\usepackage{cite}
\usepackage{mathtools, nccmath}

\usepackage[section]{placeins}

\providecommand{\Q}{\mathbb{Q}}
\providecommand{\Z}{\mathbb{Z}}

\newcommand{\Gal}{\mathrm{Gal}}
\renewcommand{\arraystretch}{2.0}

\usepackage{longtable}
\usepackage{hyperref}

\usepackage{multirow}

\let\svthefootnote\thefootnote
\newcommand\freefootnote[1]{%
  \let\thefootnote\relax%
  \footnotetext{#1}%
  \let\thefootnote\svthefootnote%
}

\begin{document}

\begin{abstract}
In this paper, we consider the equation \[ x^2 - q^{2k+1} = y^n, \qquad q \nmid x, \quad 2 \mid y, \] for integers $x,q,k,y$ and $n$, with $k \geq 0$ and $n \geq 3$. We extend work of the first and third-named authors by finding all solutions in the cases $q= 41$ and $q = 97$. We do this by constructing a Frey--Hellegouarch $\mathbb{Q}$-curve defined over the real quadratic field $K=\mathbb{Q}(\sqrt{q})$, and using the modular method with multi-Frey techniques. 
\end{abstract}

\maketitle

\section{Introduction}

The equation \begin{equation}\label{LebNag} x^2 + D = y^n \end{equation} is known as the \emph{Lebesgue--Nagell equation}. Here, $x$ and $y$ are coprime integers, $n \geq 3$ and $D$ is an integer whose prime divisors belong to a fixed finite set. The Lebesgue--Nagell equation has a rich history and many cases have been resolved through use of a wide variety of techniques, ranging from primitive divisor arguments and bounds for linear forms in logarithms, to the modular method, based upon the modularity of Galois representations attached to Frey--Hellegouarch curves.

\freefootnote{\emph{Keywords}: Lebesgue--Nagell, Elliptic curves, Frey curve, multi-Frey, $\mathbb{Q}$-curves, modularity, level-lowering,  Galois representations,  newforms.}
\freefootnote{\emph{MSC2020}: 11D41, 11D61, 11F80, 11G05.}

In recent papers of the first- and third-named authors \citep{BeSi2} and \citep{BeSi}, various tools are developed to tackle equation (\ref{LebNag}) in the two ``difficult'' cases,  where either $D>0$ and $y$ is even, or where $D < 0$. In particular,  \citep{BeSi} focusses upon these situations where, additionally, it is assumed that
$D$ has a single prime divisor. For primes $q < 100$, the only unsolved cases of the equation $x^2 \pm q^\alpha=y^n$  (see \citep[Theorem~3 and Proposition~13.3]{BeSi}) correspond to
\begin{equation}\label{maineq00}
x^2-2 = y^n,
\end{equation}
\begin{equation}\label{maineq0}
x^2 - q^{2k+1} = y^n,  \quad 2 \nmid y,
\end{equation} 
for $q \in \{ 3, 5, 17, 37, 41, 73, 89 \}$,
and
\begin{equation}\label{maineq}
x^2 - q^{2k+1} = y^n, \quad 2 \mid y, 
\end{equation} 
for $q \in \{ 17, 41, 89, 97 \}$. Here, $k$ is a nonnegative integer and, in each case, we suppose that $n \geq 3$ and that $\gcd(x,y)=1$. The fundamental obstruction to resolving equations (\ref{maineq00}) and (\ref{maineq0}), for $q \in \{ 3, 5, 17, 37 \}$, 
lies in the existence of a solution with $y=\pm 1$, valid for all (odd) exponents $n$. The analogous obstruction, in case of equation (\ref{maineq0}) with $q \in \{ 41, 73, 89 \}$, or equation (\ref{maineq}), for $q \in \{ 17, 41, 89, 97 \}$, is slightly more subtle, arising from the fact that $q \pm 8$ is square, in the first case, and from the identities
\begin{equation} \label{ident}
23^2-17=2^9, \; \; 13^2-41=2^7, \; \; 91^2-89=2^{13} \; \; \mbox{ and } \; \; 15^2-97=2^7,
\end{equation}
in the second.

In this paper, we will concentrate on equation (\ref{maineq}), 
developing new techniques to handle further values of $q$ via the use of $\Q$-curves and multi-Frey techniques, overcoming some of these obstructions. In particular, we will prove the following.

\begingroup
\renewcommand\thetheorem{1}
\begin{theorem}\label{mainthm}
Let $q \in \{ 41, 97 \}$. Then the solutions to equation (\ref{maineq}) in integers $x, y, k, n$, with $x, k \geq 0$, $n \geq 3$ and $\gcd(x,y)=1$ are as follows:
$$
\begin{array}{c}
(q,x,y,k,n) \; = \; (41,3,-2,0,5),\; (41,7,2,0,3), \; (41,13,2,0,7), \\
(41,411,10,1,5), \; (97,15,2,0,7) \; \mbox{ and } \; (97,77,18,0,3).
\end{array} 
$$
\end{theorem}
\endgroup

We are unable to provide a similar result for the cases $q = 17$ and $q = 89$, with obstructions to our method arising from the first and third identities in (\ref{ident}). We will still consider the cases $q = 17$ and $q = 89$ throughout, and in Section 5  will explain precisely why these solutions prevent us from resolving equation (\ref{maineq}) for these primes $q$. Note that Barros \citep{Barr} claims to resolve equations   (\ref{maineq0}) and (\ref{maineq}) in the case $k=0$ and $q=89$; his argument overlooks the obstructing solution corresponding to the third identity on (\ref{ident}). 

Thanks to \citep[Theorems 1, 3 and 5]{BeSi} in the case $q = 97$, we obtain the following corollary to Theorem \ref{mainthm}.

\begingroup
\renewcommand\thetheorem{2}
\begin{corollary} All solutions to the equation \[ x^2 \pm 97^\alpha = y^n, \qquad 97 \nmid x,\] for integers $x, \alpha, y$ and $n$ with $x, \alpha \geq 1$ and $n \geq 3$ are given by 
$$
(\pm 15)^2 - 97 = 2^7, \; \;  (\pm 77)^2 - 97 = 18^3, \; \; 
$$
$$
 (\pm 175784)^2 - 97^4 = 3135^3 \; \; \mbox{ and }  \; \;  (\pm 48)^2 + 97 = 7^4.
$$
\end{corollary}
\endgroup

$\Q$-curves have been successfully applied to the problem of solving  Diophantine equations in the past; the first such example is due to Ellenberg \citep{Ellenberg}, where he treats the equation
$$
x^2+y^4=z^n,
$$
for suitably large $n$. We refer to \citep{vanlangen} for a clear exposition of the general method and the references therein for more examples of this approach; we highlight \citep{BCDY}, since the set-up (once the Frey--Hellegouarch $\Q$-curve has been constructed) is most similar to ours.

We now outline the rest of the paper. In Section 2, we will associate a rational Frey--Hellegouarch curve $G$ to equation (\ref{maineq}) and recall some results from \citep{BeSi}. In Section 3, we construct a second Frey--Hellegouarch curve $E$, this time defined over the real quadratic field $\Q(\sqrt{q})$, show that it is a $\Q$-curve, and compute its conductor. Then, in Section 4, we will investigate some further properties of this $\Q$-curve, and in particular prove that its restriction of scalars is an abelian surface of $\mathrm{GL}_2$-type, which will allow us to associate the mod $n$ Galois representation of $E$ to a classical newform of a certain level and character. Finally, in Section 5, we will try and eliminate newforms to reach a contradiction.

The \texttt{Magma} \citep{magma} files used to carry out the computations in this paper are available at:

\vspace{3pt}

\begin{center}
\noindent \url{https://github.com/michaud-jacobs/Q-curves}
\end{center}

\bigskip

We would like to thank the anonymous referee for a very careful reading of the paper and many helpful comments. The second-named author would like to thank Damiano Testa for many useful discussions.

\section{A Rational Frey--Hellegouarch Curve}

Let $q \in \{17,41,89,97 \}$ and suppose that $(x,k,y,n)$ is a solution to equation (\ref{maineq}). We will assume that $x \equiv 1 \pmod{4}$ by replacing $x$ by $-x$ if necessary. We will also assume that $n$ is prime with $n \geq 7$, since the cases $n \in \{3,4,5\}$ are resolved for all values of $q$ in the range $3 \leq q < 100$ in \citep[pp.~6--7,~24]{BeSi}.
Following \citep[Proposition~14.1]{BeSi}, we associate a Frey--Hellegouarch elliptic curve, defined over $\Q$, to this solution: 
\begin{equation}\label{RatFrey}
 G = G_{x,k,q} \; \; : \; \;  Y^2 = X^3 + 4xX^2 + 4(x^2 - q^{2k+1})X. 
 \end{equation}
The conductor of $G$ is given by 
$$
N_G = q \, \mathrm{Rad}(y),
$$
where $ \mathrm{Rad}(y)$ is the product of the distinct primes dividing the nonzero integer $y$. We write $\overline{\rho}_{G,n}$ for the mod $n$ Galois representation of the elliptic curve $G$.
Applying standard level-lowering results, followed by the elimination of some newforms at level $2q$ (recall that $y$ is even),  we find that $\overline{\rho}_{G,n} \sim \overline{\rho}_{F,n} $ for $F = F_q $ an elliptic curve of conductor $2q$ given, in Cremona's notation,  in Table \ref{TabRat} (see \citep[Proposition~14.1]{BeSi}). Each curve $F$ in Table \ref{TabRat} corresponds to (at least) one solution to equation (\ref{maineq}). We have  
\begin{align*} (-23)^2 - 17 = 2^9 \quad & \text{and} \quad  G_{-23,0,17} \cong F_{17},  \\ 
13^2 - 41 = 2^7 \quad & \text{and} \quad  G_{13,0,41} \cong F_{41}, \\ (-91)^2 - 89 = 2^{13} \quad & \text{and} \quad  G_{-91,0,89} \cong F_{89},  \\  (-15)^2 - 97 = 2^7 \quad & \text{and} \quad  G_{-15,0,97} \cong F_{97}.  \\
\end{align*}
These isomorphisms of elliptic curves prevent us from using the isomorphisms of mod $n$ Galois representations $\overline{\rho}_{G,n} \sim \overline{\rho}_{F,n} $ to obtain an upper bound on $n$. We can, in fact, deduce such a bound through appeal to linear forms in logarithms, but it will be impractically large for our purposes, in each case well in excess of $10^{10}$. It is worth observing that equation (\ref{maineq}) is the more problematical case (in comparison to equation (\ref{maineq0})), for the purposes of application of bounds for linear forms in logarithms. In case of equation  (\ref{maineq0}), results of Bugeaud \citep{Bu-Acta} imply that
$$
n < 4.5 \cdot 10^6 q^2 \log^2q,
$$
which we can, with care, sharpen to an upper bound upon $n$ of somewhat less than  $10^6$ for, say, $q=3$ in equation  (\ref{maineq0}). Even with such a bound, it remains impractical to finish the problem via this approach, since we have no reasonable techniques to obtain a contradiction for a fixed value of $n$ in (\ref{maineq0}), while, as discussed in
\citep[pp.~34--35]{BeSi}, in the case of equation  (\ref{maineq}), we have such a method which is  unfortunately computationally infeasible, given the size of our upper bounds for $n$.
For small values of $n$, however, we have the following result which arises from applying the modular method with the Frey--Hellegouarch curve $G$.

\begingroup 
\renewcommand*{\arraystretch}{1.8}
\begin{table}[ht!]
\begin{center} \small
\begin{tabular}{ |c|c|c|c|c| } 
 \hline
 $q$ & $17$ & $41$ & $89$ & $97$ \\
 \hline
 $F_q$ & 34a1 & 82a1 & 178b1 & 194a1 \\
 \hline
\end{tabular}
\vskip2ex
\caption{\label{TabRat}\normalsize Elliptic curves that cannot be eliminated.}
\end{center}
\end{table} 
\endgroup  
\normalsize

\begin{lemma}[{\citep[Proposition~14.1]{BeSi}}] \label{>1000} Let $ q \in \{ 17, 41, 89, 97 \} $ and suppose that $(x,k,y,n)$ is a solution to equation (\ref{maineq}) with $x \equiv 1 \pmod{4}$ and $n \geq 7$ prime. Then $n > 1000$ or $(q,x,y,k,n)$ is one of 
$$
(17,-71,2,1,7), \;  (41,13,2,0,7), \;  (89,-91,2,0,13) \mbox{ or }  (97,-15,2,0,7).
$$
\end{lemma}

We note that \citep[Proposition~14.1]{BeSi} also provides information on the parity of the exponent $k$;  we will not have use of this.


\begin{proof} When $q \ne 17$, this follows immediately from  \citep[Proposition~14.1]{BeSi}. For $q=17$, we use exactly the same method to achieve the desired result. Using \citep[Lemma~14.3]{BeSi} deals with all $n > 7$. For the case $n = 7$, we start by applying \citep[Lemma~14.6]{BeSi}, and following the arguments of \citep[pp. 32--34]{BeSi} leaves us needing to solve three Thue--Mahler equations, each of degree $7$. To be precise, we need to solve
$$
a_7X^7+a_6X^6Y+a_5X^5Y^2+a_4X^4Y^3+a_3X^3Y^4+a_2X^2Y^5+a_1XY^6+a_0Y^7=17^k,
$$
where $(a_7,a_6,a_5,a_4,a_3,a_2,a_1,a_0)$ is one of 
$$
(139,1519,7119,18515,28945,27069,14133,3137),
$$
$$
(17, 189, 861, 2345, 3395, 3591, 1519, 467)
$$
or
$$
(1,  189, 14637,  677705, 16679635, 299923911, 2156762783, 11272244723).
$$
Using the code and techniques of \citep{ThueMahler}, we find that the first two of these equations yield no solutions, and the third has only the solution $(X,Y)=(1,0)$,  which corresponds to the identity $(-71)^2 - 17^3 = 2^7$. These computations took approximately 2000 seconds for the first equation, and just over one minute for each of the second and third, running Magma V2.24-5 on a 2019 MacBook Pro. 
\end{proof}

To proceed further, we will now turn our attention to
a new Frey--Hellegouarch curve, defined over the real
quadratic field $\Q(\sqrt{q})$.

\section{Constructing a Frey--Hellegouarch \texorpdfstring{$\Q$}{}-Curve}

Let $q \in \{17,41,89,97\}$ and write $M = \Q(\sqrt{q})$. In this section, we construct a new Frey--Hellegouarch curve, this time defined over $M$. This curve will be a $\Q$-curve, i.e. an elliptic curve, defined over some number field, that is isogenous over $\overline{\Q}$ to all of its Galois conjugates. The $\Q$-curve we define will in fact be \emph{completely defined} over $M$, meaning that the isogeny between the curve and its conjugate is also defined over $M$. We will start by following the approach suggested in \citep[pp.~47--48]{BeSi}.

In each case $M$ has class number $1$. We write $\mathcal{O}_M$ for the ring of integers of $M$. We write $\sigma$ for the non-trivial element of $\Gal(M / \Q)$, and for $z \in M$ we will use the notation $\overline{z} = \sigma(z)$. Although we may suppose that $n > 1000$ by Lemma \ref{>1000}, we will for the moment simply assume $n \geq 7$ (and $n$ prime) as in the previous section.  We will write $\delta$ for a fundamental unit for $\mathcal{O}_M$.  For each value of $q$, the rational prime $2$ splits in $\mathcal{O}_M$.
Let 
\[ \gamma = \begin{cases*} (-3 + \sqrt{q})/2 & if $q = 17$, \\ (-19 - 3 \sqrt{q}) /2 & if $q = 41$, \\ (9 + \sqrt{q})/2 & if $q = 89$, \\ (325 + 33 \sqrt{q})/2 & if $q = 97$.   \end{cases*} 
\] 
Here, we have chosen $\gamma$ such that it is a generator for one of the two prime ideals above $2$, and such that 
$$
 \gamma \overline{\gamma} = -2, \; \;  \overline{\gamma} \equiv -1 \pmod{ \gamma^2} \; \; \mbox{ and } \; \;  \sqrt{q} \equiv -1  \pmod{\gamma^2}. 
 $$
The relevance of these properties will be seen in due course. 

We will now factor the left-hand side of equation (\ref{maineq}) over $M$. Writing $y = 2y_1$, we have \[ \left( \frac{x+q^k \sqrt{q} }{2} \right) \left( \frac{x - q^k \sqrt{q} }{2} \right) = 2^{n-2} y_1^n. \] Since $ q \equiv 1 \pmod{4}$, each factor on the left-hand side is in $\mathcal{O}_M$. Since $q \nmid x$, we see that \begin{equation}\label{coprim} \gcd \left( \frac{x+q^k \sqrt{q} }{2}, \frac{x-q^k \sqrt{q} }{2}  \right) = 1. \end{equation} Now, because $\overline{\gamma} \equiv -1 \pmod{ \gamma^2}$ and $ x \equiv 1 \pmod{4}$, we see that $\gamma$ must divide $(x + q^k \sqrt{q})/2$, and so $\overline{\gamma}$ will divide  $(x - q^k \sqrt{q})/2$. Then by coprimality of the two factors, we have \begin{equation*}  \frac{x+q^k \sqrt{q} }{2}  = \delta^r \gamma^{n-2} \alpha^n,
\end{equation*} for some $r \in \mathbb{\Z}$ and $\alpha \in \mathcal{O}_M$. We then obtain that \begin{equation}\label{newnn2}
q^k\sqrt{q} = \delta^r \gamma^{n-2} \alpha^n - \overline{\delta}^r \overline{\gamma}^{n-2} \overline{\alpha}^n.
\end{equation} Treating this equation as a generalized Fermat equation of signature $(n,n,n)$ with solution $(\alpha, \overline{\alpha},1)$ would lead to a Frey--Hellegouarch curve isogenous to the rational Frey--Hellegouarch curve $G$ defined by (\ref{RatFrey}). Instead, we will view this as an equation of signature $(n,n,2)$.

Write $k = 2m$ or $2m+1$ according to whether $k$ is even or odd. Let \begin{equation*}
w = \begin{cases*} \dfrac{(x+q^{2m}\sqrt{q})}{2} \cdot  \sqrt{q}^3  = \delta^r \gamma^{n-2} \alpha^n \sqrt{q}^3 & if $k = 2m$, \\
 \dfrac{(x+q^{2m+1}\sqrt{q})}{2} \cdot \sqrt{q} =  \delta^r \gamma^{n-2} \alpha^n \sqrt{q} & if $k = 2m+1$. \end{cases*}
\end{equation*} From equation (\ref{newnn2}), we deduce that \begin{equation*} \gcd(w, \overline{w}) = \begin{cases*} \sqrt{q}^3 & if $k = 2m$, \\ \sqrt{q} & if $k = 2m+1$. \end{cases*} 
\end{equation*} We also have \begin{equation*} w + \overline{w} = q^{2m+2}. \end{equation*}
One can attach to any equation of the form $w+\overline{w} = u^2$, with $u \in \Q$, a Frey--Hellegouarch $\Q$-curve; see, by way of example,  \citep[pp.~199,~203--204]{vanlangen}. We take our $\Q$-curve to be \begin{equation}\label{E}
E = E_{x,m} = E_{x,m,q}: \; Y^2 = X^3 + 2\gamma q^{m+1}X^2 + \gamma^2 w X.
\end{equation} This $\Q$-curve is a quadratic twist by $\gamma$ of the $\Q$-curve one would obtain applying the recipe in \citep[p.~199]{vanlangen}. The reason for twisting by $\gamma$ is to ensure the curve $E$ is \emph{completely defined} over $M$, meaning the isogeny between $E$ and its Galois conjugate is also defined over $M$. We have 
\begin{equation} 
\overline{E} = \overline{E}_{x,m} = \overline{E}_{x,m,q}: \; Y^2 = X^3 + 2\overline{\gamma} q^{m+1}X^2 + \overline{\gamma}^2 \overline{w} X,
\end{equation} and a $2$-isogeny, defined over $M$, \begin{equation}\label{varphi} \varphi_\sigma: \overline{E} \rightarrow E,  \qquad (X,Y) \mapsto \left( \frac{X^2 + 2 \overline{\gamma}q^{m+1}X + \overline{\gamma}^2 \overline{w}}{\overline{\gamma}^2 X} , \frac{(X^2 -\overline{\gamma}^2 \overline{w}) Y }{\overline{\gamma}^3 X^2} \right). \end{equation}

We would like to compute the conductor $\mathcal{N}_E$ of $E$. We first note that the curve $E$ has the following standard invariants : 
\begin{equation*} c_4  = \gamma^6 \overline{\gamma}^4(w + 4 \overline{w}), \; \;  c_6 = \gamma^9 \overline{\gamma}^6 (w  - 8 \overline{w}) q^{m+1} \; \; \mbox{ and } \; \;  \Delta = \gamma^{12} \overline{\gamma}^6 w^2 \overline{w}. 
\end{equation*}

\begin{lemma}\label{Conductor1} Let $n \geq 11$. The curve $E$ has multiplicative reduction at both primes of $M$ above $2$. As a consequence, $E$ does not have complex multiplication.
\end{lemma}

\begin{proof} We recall that $\gamma$ and $\overline{\gamma}$ generate the two prime ideals of $M$ above $2$. The model $E$ is not minimal at these primes, but we will not actually need to write down a minimal model.

We start by noting that $\ord_\gamma(\alpha) = \ord_{\overline{\gamma}}(\overline{\alpha})$. Using (\ref{coprim}), we also see that $\ord_{\overline{\gamma}}(\alpha) = \ord_\gamma(\overline{\alpha}) = 0$. We then have 
$$
\ord_\gamma(w) = \ord_{\overline{\gamma}} (\overline{w}) = n - 2 + n \ord_\gamma (\alpha) ~\text{ and }~  \ord_{\overline{\gamma}}(w) = \ord_{\gamma} (\overline{w})=0,
$$
whence we deduce that 
\begin{align*} \ord_\gamma(c_4) & = 6 + \ord_\gamma(w+4\overline{w}) = 8, \\ \ord_\gamma(c_6) & = 9 + \ord_\gamma(w-8\overline{w}) = 12, \\ \ord_\gamma(\Delta) & = 12 + 2(n-2+n \ord_\gamma(\alpha)) = 8 + 2n(1+\ord_\gamma(\alpha)). 
\end{align*} 
Similarly, we see that 
$$
\ord_{\overline{\gamma}}(c_4) = 4, \; \;   \ord_{\overline{\gamma}}(c_6) = 6 \; \; \mbox{ and } \; \;   \ord_{\overline{\gamma}}(\Delta) = 4 + n(1+\ord_{\overline{\gamma}}(\overline{\alpha}) ).
$$
Writing $j = c_4^3 / \Delta$ for the $j$-invariant of $E$, we have that \[ \ord_\gamma(j) = 16 - 2n(1+\ord_\gamma(\alpha)) < 0 ~ \text{ and } ~ \ord_{\overline{\gamma}}(j) = 8 - n(1+\ord_\gamma(\alpha)) < 0,  \] since $n \geq 11$ by assumption. We note that these inequalities will in fact hold whenever $n \geq 9$. We conclude that $E$ has potentially multiplicative reduction at each prime above $2$. We can in fact already see at this point that $E$ does not have complex multiplication, since the $j$-invariant of $E$ is non-integral.

In order to show that $E$ has multiplicative reduction at each prime above $2$, it will be enough to prove that the extension $M(\sqrt{-c_6/c_4})/M$ is unramified at $\gamma$ and $\overline{\gamma}$ (see \citep[Lemma~4.3]{small_real} for example). 
We have, recalling that $\gamma \overline{\gamma} = -2$, 
\begin{align*} - \frac{c_6}{c_4} & = - \gamma^3 \overline{\gamma}^2 \cdot \frac{w-8\overline{w}}{w+4\overline{w}} \cdot  \sqrt{q}^{2m+2} = -\gamma^3 \overline{\gamma}^2 \cdot \frac{w + \gamma^3 \overline{\gamma}^3 \overline{w}}{w+ \gamma^2 \overline{\gamma}^2 \overline{w}} \cdot \sqrt{q}^{2m+2}
 \\ & = -\frac{\gamma^4}{\gamma} \overline{\gamma}^2 \cdot \frac{w / \gamma^3 + \overline{\gamma}^3 \overline{w}}{w / \gamma^3+ \overline{\gamma}^2  \overline{w} / \gamma } \cdot \sqrt{q}^{2m+2} = - (\gamma^2 \overline{\gamma} \sqrt{q}^{m+1})^2 \cdot \frac{w / \gamma^3 + \overline{\gamma}^3 \overline{w}}{w / \gamma^2+ \overline{\gamma}^2  \overline{w} }. \end{align*} Write\[ \eta = \gamma^2 \overline{\gamma} \sqrt{q}^{m+1} \quad \text{ and } \quad \kappa =  - \frac{w / \gamma^3 + \overline{\gamma}^3 \overline{w}}{w / \gamma^2+ \overline{\gamma}^2  \overline{w} }, \] so that $M( \sqrt{-c_6 / c_4 }) = M(\sqrt{\kappa}) = M((1+\sqrt{\kappa})/2)$. 

Consider the numerator of $\kappa$. We have that $\ord_\gamma(\overline{\gamma}^3\overline{w}) = 0$, and 
$$
\ord_\gamma( w / \gamma^3) = n - 2 + n \ord_\gamma(\alpha) - 3  = n - 5 + n \ord_\gamma(\alpha) \geq 6 > 0,
$$
as $n \geq 11$, so $\gamma$ does not divide the numerator and, similarly for the denominator. So $\ord_\gamma(\kappa) = 0$, and similarly, $\ord_{\overline{\gamma}}(\kappa) = 0$. We have that $\ord_\gamma(w / \gamma^3), \; \ord_\gamma(w/\gamma^2) > 2$, so $\kappa \equiv -\overline{\gamma}  \equiv 1 \pmod{\gamma^2}$ by our choice of $\gamma$. We also have that $\kappa \equiv -1/\gamma \equiv 1 \pmod{\overline{\gamma}^2}$ since $\gamma \equiv -1 \pmod{\overline{\gamma}^2}$.

Now, $(1+\sqrt{\kappa})/2$ satisfies the polynomial \[ X^2 - X + \frac{1-\kappa}{4}.\] This polynomial has discriminant $\kappa$ and is integral at $\gamma$ and $\overline{\gamma}$. This proves that the extension $M(\sqrt{-c_6/c_4})/M$ is unramified at $\gamma$ and $\overline{\gamma}$. \end{proof}


\begin{lemma}\label{Conductor2} Let $n \geq 11$.
\begin{enumerate}
\item If $\pi \nmid 2q \alpha \overline{\alpha}$ is a prime of $M$, then $E$ has good reduction at $\pi$;
\item If $\pi \nmid 2q$ is a prime of $M$ dividing $\alpha$ or $\overline{\alpha}$, then the model of $E$ is minimal at $\pi$, the prime $\pi$ is of multiplicative reduction for $E$, and $n \mid \ord_\pi(\Delta)$;
\item $E$ has additive, potentially good reduction at $\sqrt{q} \cdot \mathcal{O}_M$. In particular, we have that $\ord_{\sqrt{q}}(\mathcal{N}_E) = 2$, since $q \geq 5$.
\end{enumerate} 
\end{lemma}

\begin{proof}
Let $\pi \nmid 2q$ be a prime of $M$. So $\pi \nmid \gamma \overline{\gamma} \sqrt{q}$. If $\pi \nmid \alpha \overline{\alpha}$, then $\pi \nmid \Delta$, so $\pi$ is a prime of good reduction for $E$, proving the first part.

Suppose instead that $\pi \nmid 2q$, but that $\pi \mid \alpha \overline{\alpha}$. Then $\pi \mid \Delta$. By (\ref{coprim}), we see that $\gcd(\alpha, \overline{\alpha}) = 1$. So either $\pi \mid  \alpha$ or $\pi \mid \overline{\alpha}$, but not both. So $\pi \mid w$ or $\pi \mid \overline{w}$, but not both. It follows that $\pi \nmid c_4$. So $E$ is minimal at $\pi$, and $\pi$ is a prime of multiplicative reduction for $E$. Moreover, $\ord_\pi(\Delta) = 2n \ord_\pi(\alpha) + n \ord_\pi(\overline{\alpha}) \equiv 0 \pmod{n}$, as required.

Finally, we consider $\sqrt{q}$. We have that $\ord_{\sqrt{q}}(w) = \ord_{\sqrt{q}}(\overline{w}) = 1$ or $3$ according to whether $k$ is odd or even. So $\ord_{\sqrt{q}}(\Delta) = 3$ or $9$, and $\sqrt{q} \mid c_4$. It follows that $E$ is minimal with additive reduction at $\sqrt{q}$. To see that we have potentially good reduction, we show that $\ord_{\sqrt{q}}(j)  \geq 0$. We must show that $3\ord_{\sqrt{q}}(c_4) \geq \ord_{\sqrt{q}}(\Delta)$, and this inequality holds since \[\left(\ord_{\sqrt{q}}(\Delta), \ord_{\sqrt{q}}(c_4) \right) = \begin{cases} (9, \geq 3) &  \text{if $k$ is even,} \\ (3, \geq 1) & \text{if $k$ is odd}. \\ \end{cases} \] \end{proof}

Combining Lemmas \ref{Conductor1} and \ref{Conductor2}, we have that \[ \mathcal{N}_E = \left( \gamma \overline{\gamma} \cdot \sqrt{q}^2 \cdot  \mathrm{Rad}_2(\alpha \overline{\alpha}) \right) \cdot \mathcal{O}_M, \] where $\mathrm{Rad}_2(\alpha \overline{\alpha})$ denotes the product of all prime ideals of $M$ dividing $\alpha \overline{\alpha}$ but not dividing $2$. 

\section{Irreducibility and Level-Lowering}

We would like to apply certain level-lowering results to $E$ in order to relate $E$ to a newform of a particular level and character. We must first prove irreducibility of $\overline{\rho}_{E,n}$, the mod $n$ Galois representation of $E$. 
We highlight the fact that we will use the rational Frey--Hellegouarch curve $G$ to help us prove the irreducibility of $\overline{\rho}_{E,n}$.

\begin{proposition}\label{irreduc} Let $q \in \{17,41,89,97 \}$. The representation $\overline{\rho}_{E,n}$ is irreducible for $n \geq 11$.
\end{proposition}

\begin{proof} Suppose that $\overline{\rho}_{E,n}$ is reducible with $n \geq 11$. If $n = 13$, then arguing as in \citep[p.~215]{vanlangen}, $E$ would give rise to a $\Q(\sqrt{13})$-point on the modular curve $X_0(26)$, a contradiction, since $q \ne 13$. We will therefore suppose that $n = 11$ or $n > 13$. Since $E$ is a $\Q$-curve defined over a quadratic field and the isogeny $\varphi_\sigma$ has degree $2$, \citep[Proposition~3.2]{Ellenberg} tells us that every prime of $M$ of characteristic $>3$ is a prime of potentially good reduction for $E$. 

We first show that $y$ must be a power of $2$. If $\ell >3$ is a prime with $\ell \mid y$, then each prime of $M$ above $\ell$ will divide either $\alpha$ or $\overline{\alpha}$, and it follows (by Lemma \ref{Conductor2}) that we have a prime of characteristic $\ell > 3$ of multiplicative reduction for $E$, a contradiction. Next, suppose that $3 \mid y$. Then $3$ is a prime of multiplicative reduction for the rational Frey--Hellegouarch curve $G$ defined in (\ref{RatFrey}). From the isomorphism $\overline{\rho}_{G,n} \sim \overline{\rho}_{F,n}$, for $F$ an elliptic curve of level $2q$ in Table \ref{TabRat}, we have, writing $f$ for the newform corresponding to $F$, that 
\begin{equation} \label{Hasse!}
n \mid 3+1 \pm a_3(f).
\end{equation}
From the Hasse bound, we have that $\abs{a_3(f)} \leq 2 \sqrt{3}$ and hence the right-hand-side of (\ref{Hasse!}) is a nonzero integer, bounded above by $4+2\sqrt{3}$. This contradicts $n \geq 11$ and so we may conclude that $y$ is necessarily a power of $2$, say $y = 2^s$, with $ s \geq 1$ since $y$ is even.

We thus have that $x^2 - 2^{ns} = q^{2k+1}$. By \citep[p.~328]{Ivorra}, this equation has no solutions with $n \geq 11$, provided  $2k+1 >1$. It follows that  $k = 0$ and we have \[ x^2 = 2^{ns} + q. \] Multiplying both sides by $2^2$ or $2^4$ if necessary, we obtain an integral point on one of the following elliptic curves: 
\begin{align*}
Y^2  = X^3 + q, \; \; 
Y^2  = X^3 + 2^2q \; \; \mbox{ or } \; \; 
Y^2  = X^3 + 2^4 q .
\end{align*}
Computing the integral points on each of these curves for each value of $q$ using \texttt{Magma} quickly leads to a contradiction.
\end{proof}

\begin{remark} At this point, we could apply standard level-lowering results over $M$ (see \citep[Theorem~7]{asym} for example) to relate $\overline{\rho}_{E,n}$ to the Galois representation of a Hilbert newform at level $\gamma \overline{\gamma} \sqrt{q}^2 \cdot \mathcal{O}_M$. For $q = 17, 41, 89,$ and $97$, the dimensions of these spaces of newforms are $46, 1093, 9631,$ and $26378$ respectively. Computing the newform data at these levels using \texttt{Magma} is certainly possible for $q = 17$, and would also likely be achievable for $q = 41$ by working directly with Hecke operators (see \citep[p.~342--343]{MJ} for example), but for $q = 89$, and especially for $q = 97$, the dimensions are likely too large for current computations. The $\Q$-curve approach we now present will allow us to work with classical modular forms and make the resulting computations feasible.
\end{remark}

We start by computing some data associated to the $\Q$-curve $E$, which we recall does not have complex multiplication (by Lemma \ref{Conductor1}). We will use the notation and terminology of Quer \citep{Quer}. We note that we are in a similar set-up to that of \citep[pp.~8--9]{BCDY}. As in the previous section, we write $\mathrm{Gal}(M/\Q) = \{ 1, \sigma \}$. We have the isogeny $\varphi_\sigma : \overline{E} \rightarrow E$ given by (\ref{varphi}), and $\varphi_1$ will denote the identity morphism on $E$. Write $c: \Gal(M/\Q) \rightarrow \Q^*$ for the $2$-cocycle given by \[ c(s,t) = \varphi_s \, {}^s\varphi_t \; \varphi_{st}^{-1}. \] We have that $c(1,1) = c(1,\sigma) = c(\sigma, 1) = 1$. By a direct computation with \texttt{Magma}, we verify that $c(\sigma,\sigma) = \varphi_\sigma ({}^\sigma \varphi_\sigma) = -2$.

Next, define  \[ \beta : \mathrm{Gal}(M / \Q) \rightarrow \overline{\Q}^*, \qquad \beta(1) = 1,~ \beta(\sigma) = \sqrt{-2}. \] This map satisfies  \begin{equation}\label{schur} c(s,t) = \beta(s)\beta(t) \beta(st)^{-1} \quad \text{ for } s,t \in \Gal(M/ \Q). \end{equation} It follows that $\beta$ is a \emph{splitting map} for $c$ (as defined in \citep[p.~298]{Quer}). The \emph{splitting character} associated to $\beta$ is then defined by \[ \epsilon(s) = \beta(s)^2  / \deg(\varphi_s). \] So $\epsilon(1) = 1$ and $\epsilon(\sigma) = -1$, and $\epsilon$  is the quadratic Galois character associated to $M$. Since $q \equiv 1 \pmod{4}$, we have $M \subset \Q(\zeta_q)$, and we may also view $\epsilon$ as a quadratic Dirichlet character $\epsilon: (\Z / q\Z)^\times \rightarrow \{\pm1 \}$  of conductor $q$ via  $(\Z/q\Z)^\times \cong \mathrm{Gal}(\Q(\zeta_q)/ \Q) \rightarrow \Gal(M/ \Q)$.

Write $B = \mathrm{Res}_\Q^M(E)$ for the restriction of scalars of $E$ to $\Q$. This is an abelian surface defined over $\Q$ and plays an important role. The relation (\ref{schur}) shows that the $2$-cocycle $c$ has trivial Schur class (i.e. is trivial when viewed as an element of $H^2(\mathrm{Gal}(\overline{\Q}/\Q), \overline{\Q}^*)$ with trivial action). By \citep[Proposition~5.2]{Quer}, we deduce that $B$ decomposes as a product of abelian varieties of $\mathrm{GL}_2$-type. Moreover, the $\Q$-simple abelian variety of $\mathrm{GL}_2$-type, $A_\beta$, attached to $\beta$, which is a quotient of $B$, will have endomorphism algebra $\Q(\beta(1),\beta(\sigma)) = \Q(\sqrt{-2})$ (see \citep[pp.~305--306]{Quer}), and is therefore itself an abelian surface. It follows that $B$ is $\Q$-isogenous to $A_\beta$, so $B$ is $\Q$-simple and of $\mathrm{GL}_2$-type with $\Q$-endomorphism algebra $\Q(\sqrt{-2})$. We record this in the following proposition.
 
\begin{proposition} The abelian surface $B  = \mathrm{Res}_\Q^M(E)$ is $\Q$-simple and of $\mathrm{GL}_2$-type. It has  $\Q$-endomorphism algebra $\Q(\sqrt{-2})$. The conductor of $B$ is given by \[ N_B = \left(2q^2 \,\mathrm{Rad}_2(y) \right)^2. \]
\end{proposition}

\begin{proof} 
It remains to compute the conductor of $B$. This can be obtained from the conductor of $E$ using the formula in \citep[Proposition~1]{Milne}. Writing $\Delta_M$ for the discriminant of $M$, we have \[ N_B = (\Delta_{M})^2 \, \mathrm{Norm}(\mathcal{N}_E) = q^2 \cdot 2^2 q^2 \cdot \mathrm{Norm}(\mathrm{Rad}_2(\alpha \overline{\alpha})) = 2^2 q^4 (\mathrm{Rad}_2(y))^2, \] and the proposition follows.
\end{proof}

We can now use the modularity of $B$ and standard level-lowering results to deduce the following result. 

\begin{proposition}\label{repiso} Let $q \in \{17,41,89,97\}$ and let $n \geq 11$. Write $G_M = \mathrm{\Gal}(\overline{\Q}/M)$. Then we have \begin{equation}\label{repisoeq} \overline{\rho}_{E,n} \sim \left.{\overline{\rho}_{f,\mathfrak{n}}}\right|_{G_M}, \end{equation} for $f$ a newform of level $2q^2$ and character $\epsilon$, and $\mathfrak{n}$ a prime above $n$ in the coefficient field of $f$.
\end{proposition}

\begin{proof} By \citep[Theorem~4.4]{Ribet}, $B$ is isogenous to a factor, $A_g$, of $J_1(N)$ for some $N$, where $A_g$ is the abelian variety attached to some newform $g$. We have that $N^{\dim(A_g)} = N_B = \left(2q^2 \,\mathrm{Rad}_2(y) \right)^2,$ and so $N = 2q^2 \,\mathrm{Rad}_2(y)$. Moreover, $g$ has character $\epsilon^{-1} = \epsilon$, since $\epsilon$ has order $2$.

By Proposition \ref{irreduc}, the representation $\overline{\rho}_{E,n}$ is irreducible, so the representation $\overline{\rho}_{g,\pi}$ is too, and applying standard level-lowering results, we have that $\overline{\rho}_{g,\mathfrak{\pi}} \sim \overline{\rho}_{f,\mathfrak{n}}$, for $f$ a newform of level $2q^2$ and character $\epsilon$, and $\pi, \mathfrak{n}$ primes above $n$. Since $\left.\beta\right|_{G_M}$ is trivial, using \citep[pp.~210--211]{vanlangen}, we have \[ \overline{\rho}_{E,n} \sim \left.{\overline{\rho}_{g,\mathfrak{\pi}}}\right|_{G_M}  \sim \left.{\overline{\rho}_{f,\mathfrak{n}}}\right|_{G_M},\] as required. 
\end{proof}

\section{Eliminating Newforms}

We start by using \texttt{Magma} to compute the Galois conjugacy classes of newforms (i.e. their $q$-expansions) at level $2q^2$ with character $\epsilon$. Table \ref{TabNew} records some of this data.
\begingroup 
\renewcommand*{\arraystretch}{1.8}
\begin{table}[ht!]
\begin{center} \footnotesize
\begin{tabular}{ |c|c|c|c|c|c| } 
 \hline
 $q$ & $\dim$ & no. classes & (size of class, multiplicity) & time \\
 \hline 
 $17$ & $22$ & $6$ & $(2,3), (4,1), (6,2)$ & $1$s \\ 
 \hline
 $41$ & $136$ & $18$ & $(2,4),(4,5),(6,2),(8,4),(16,1),(24,2)$ & $8$s\\ 
 \hline
 $89$ & $652$ & $26$ & \makecell{$(2,4),(4,2), (6,4), (8,3), (12,2), (24,3), (30,1), $ \\ $ (40,2),  (50,1), (60,1), (80,1), (96,2) $} & $400$s \\
 \hline
 $97$ & $774$ & $29$ & \makecell{$(2,4),(4,3), (6,3),(8,4),(12,3),(20,3),(24,1), $ \\ $ (32,3),(40,1), (48,1),(64,1),(168,2) $ }  & $739$s\\
 \hline
\end{tabular}
\vskip2ex
\caption{\label{TabNew}\normalsize Newform data. Here, \emph{dim} refers to the dimension of the space of newforms and \emph{time} refers to the computation time using a 2200MHz AMD Opterons.}
\end{center}
\end{table} 
\endgroup  
\normalsize

Let $\mathfrak{p} \nmid 2qn$ be a prime of $M$ above a rational prime $p$ and denote by $\mathrm{Frob}_\mathfrak{p} \in G_M$ a Frobenius element at $\mathfrak{p}$. Let $f$ denote the newform related to $E$ in Proposition \ref{repiso}. Then, taking traces in (\ref{repisoeq}), we have \begin{equation}\label{Traceeq} \mathrm{Tr}(\overline{\rho}_{E,n}(\mathrm{Frob}_\mathfrak{p})) =  \mathrm{Tr}(\overline{\rho}_{f,\mathfrak{n}}(\mathrm{Frob}_\mathfrak{p})).\end{equation}

We first consider the right-hand side of (\ref{Traceeq}). Writing $a_p(f)$ for the $p$-th coefficient of the $q$-expansion of $f$, we start by defining the quantity \[ t_{f, \mathfrak{p}} =  \begin{cases*} a_p(f) & if $p$ \text{splits in} $M$, \\ a_p(f)^2 + 2p  & if $p$ \text{is inert in} $M$. \end{cases*} \] By \citep[pp.~217--219]{vanlangen} for example, we have $\mathrm{Tr}(\overline{\rho}_{f,\mathfrak{n}}(\mathrm{Frob}_\mathfrak{p})) \equiv t_{f, \mathfrak{p}} \pmod{\mathfrak{n}}$, where we have used the fact that $\epsilon(p) = -1$ when $p$ is inert in $M$. We highlight the fact that $t_{f, \mathfrak{p}}$ is independent of $\mathfrak{n}$.

Next, for the left-hand side of (\ref{Traceeq}), the quantity  $\mathrm{Tr}(\overline{\rho}_{E,n}(\mathrm{Frob}_\mathfrak{p})) $ is dependent on our choice of $x$ and $m$ (i.e. dependent on our original solution to equation (\ref{maineq})). However, looking at how $E = E_{x,m}$ is defined in (\ref{E}), we see that the trace will only depend on $x$ and $q^{m} \pmod{p}$. In particular, it will only depend on the value of $x$ modulo $p$, and  $m$ modulo  $(p-1)$ (in fact it will only depend on $m$ modulo the multiplicative order of $q \pmod{p}$).
Given $0 \leq \chi \leq p-1$ and $0 \leq \mu \leq p-2$, write $E_{\chi,\mu}$ for the curve obtained by substituting $x = \chi$ and $m = \mu$ into $E_{x,m}$, defined in (\ref{E}). If $\mathfrak{p} \mid \Delta_{E_{\chi,\mu}}$ then, as in the proof of Lemma \ref{Conductor2} (2), we see that $\mathfrak{p} \nmid c_4(E_{\chi,\mu})$ (and also $\mathfrak{p} \nmid c_6(E_{\chi,\mu})$), so $E_{\chi,\mu}$ has multiplicative reduction at $\mathfrak{p}$. We then have
\begin{equation*} \mathrm{Tr}(\overline{\rho}_{E_{\chi,\mu},n}(\mathrm{Frob}_\mathfrak{p})) = \begin{cases} a_{\mathfrak{p}}(E_{\chi,\mu})  & \text{if }~ \mathfrak{p} \nmid \Delta_{E_{\chi,\mu}},  \\ \mathrm{Norm}(\mathfrak{p}) + 1 & \text{if }~ \mathfrak{p} \mid \Delta_{E_{\chi,\mu}} \text{ and } (-c_6/c_4 ~ \mathrm{mod}~\mathfrak{p}) \in (\mathbb{F}_\mathfrak{p}^*)^2,  \\ -\mathrm{Norm}(\mathfrak{p})-1 & \text{if }~ \mathfrak{p} \mid \Delta_{E_{\chi,\mu}} \text{ and } (-c_6/c_4 ~ \mathrm{mod}~\mathfrak{p}) \notin (\mathbb{F}_\mathfrak{p}^*)^2. \end{cases} 
\end{equation*}
We can now simply run through all possible pairs $\chi$ and $\mu$ in this range. Define \[ \mathcal{A}_\mathfrak{p} = \{ \mathrm{Tr}(\overline{\rho}_{E_{\chi,\mu},n}(\mathrm{Frob}_\mathfrak{p})) : 0 \leq \chi \leq p-1, \quad 0 \leq \mu \leq p-2 \}. \] Then we know that $\mathrm{Tr}(\overline{\rho}_{E_{x,m},n}(\mathrm{Frob}_\mathfrak{p})) \in \mathcal{A}_\mathfrak{p}$, and we can compute the set $\mathcal{A}_\mathfrak{p}$ for any $\mathfrak{p} \nmid 2qn$.

Define \[ \mathcal{B}_{f,\mathfrak{p}}  = p \cdot \mathrm{Norm} \big(\prod_{ a \in \mathcal{A}_\mathfrak{p}} (a - t_{f,\mathfrak{p}}) \big). \] Then by (\ref{Traceeq}) we have that $n \mid \mathcal{B}_{f,\mathfrak{p}}$ whenever $\mathfrak{p} \nmid 2q$. Note that we have included a factor of $p$ in the definition of $\mathcal{B}_{f,\mathfrak{p}}$, as we would usually require $\mathfrak{p} \nmid 2qn$, but $n$ is unknown. Then if $\mathcal{B}_{f,\mathfrak{p}}$ is non-zero, we obtain a bound on $n$. Moreover, we can repeat this with many auxiliary primes $\mathfrak{p}$. If $\mathfrak{p}_1, \dots, \mathfrak{p}_r$ are primes not dividing $2q$, then \[ n \mid \mathcal{B}_f = \mathcal{B}_{f,\mathfrak{p}_1, \dots, \mathfrak{p}_r} = \gcd \left(\mathcal{B}_{f,\mathfrak{p}_1}, \dots, 
\mathcal{B}_{f,\mathfrak{p}_r} \right). \]

\begin{proof}[Proof of Theorem \ref{mainthm}]
Let $q = 41$ or $97$. We computed the value $\mathcal{B}_f$, and in particular its prime factors, for each newform $f$ at level $2q^2$ and character $\epsilon$. For most newforms $f$, we did this by choosing a prime of $M$ above each rational prime between $3$ and $30$. For computational reasons, when $q = 97$ and $f$ is one of the two newforms with coefficient field of degree $168$, we only worked with a prime above each of $3$ and $11$. We found that for each newform $f$, all prime factors of $\mathcal{B}_f$ were $<300$, except for two newforms when $q= 41$, which we denote $g_1$ and $g_2$. Since we can take $n > 1000$ by Lemma \ref{>1000}, this eliminates all newforms except for $g_1$ and $g_2$. We are unable to eliminate these two newforms as their $\mathcal{B}$ values are $0$, and this remains the case when using more auxiliary primes.

Since we managed to eliminate all newforms when $q=97$, this proves Theorem \ref{mainthm} in the case $q = 97$. For $q = 41$, we are able to eliminate the remaining forms using a multi-Frey approach. 

Let $q = 41$. Recall that $k = 2m$ or $2m+1$ according to whether $k$ is even or odd, respectively. From Table \ref{TabRat}, we know that $\overline{\rho}_{G_{x,k,q},n} \sim \overline{\rho}_{F,n}$ where $F$ is the elliptic curve with Cremona label `82a1'. Let $p = 7$ which is inert in $M$, and write $\mathfrak{p} = p \cdot \mathcal{O}_M$ for the unique prime of $M$ above $7$. 

We compute $a_7(F) = -4$. Given $0 \leq \chi \leq 6$ and $0 \leq \kappa \leq 5$, write $G_{\chi,\kappa}$ for the curve obtained by substituting $x = \chi$ and $k = \kappa$ into the definition of $G_{x,k,q}$ in (\ref{RatFrey}). Then we  compute $\mathrm{Tr}(\overline{\rho}_{G_{\chi,\kappa},n}(\mathrm{Frob}_7))$ for each $\chi$ and $\kappa$. We found this trace to be independent of $\kappa$. The traces are recorded in Table \ref{Tab41} and we see that this forces $x \equiv 6 \pmod{7}$.

\begingroup 
\renewcommand*{\arraystretch}{1.8}
\begin{table}[ht!]
\begin{center} \small
\begin{tabular}{ |c|c|c|c|c|c|c|c| } 
 \hline
 $\chi$ & $0$ & $1$ & $2$ & $3$ & $4$ & $5$ & $6$ \\
 \hline
 $\mathrm{Tr}(\overline{\rho}_{G_{\chi,\mu},n}(\mathrm{Frob}_7))$ & $0$ & $4$ & $2$ & $2$ & $ -2$ & $ -2$ & $-4$ \\
 \hline
\end{tabular}
\vskip2ex
\caption{\label{Tab41}\normalsize Proof of Theorem \ref{mainthm}: Traces of Frobenius at $7$ for $G$.}
\end{center}
\end{table} 
\endgroup  
\normalsize

When $\chi = 6$, we find that $\mathrm{Tr}(\overline{\rho}_{E_{6,\mu},n}(\mathrm{Frob}_\mathfrak{p})) = 6$ for each $0 \leq \mu \leq 5$ and $k$ even or odd. However,  \begin{align*} \mathrm{Tr}(\overline{\rho}_{g_1,\mathfrak{n}_1}(\mathrm{Frob}_\mathfrak{p}))  & = a_p(g_1)^2 + 2p = -4,  ~~\text{and} \\ \mathrm{Tr}(\overline{\rho}_{g_2,\mathfrak{n}_2}(\mathrm{Frob}_\mathfrak{p})) & = a_p(g_2)^2 + 2p = 14. \end{align*} It follows that $n \mid 7 \cdot 10$ or $ n \mid 7 \cdot 12$. So $x \not\equiv 6 \pmod{7}$, a contradiction. This completes the proof of the theorem.
\end{proof}

When $q = 17$ or $q = 89$, we found that, in each case,  there was a single obstructing newform that we were unable to eliminate. When $q = 17$, this is due to the solution $(-23)^2 - 17 = 2^9$, and when $q = 89$, this is a consequence of the identity $(-91)^2 - 89 = 2^{13}$. The exponent $n$ in each case exceeds $8$, and it follows that the curves $E_{-23,0,17}$ and $E_{-91,0,89}$ have multiplicative reduction at the primes of $M$ above $2$. This can be verified directly or can be seen from the proof of Lemma \ref{Conductor1}. We can check that the traces of Frobenius for these two curves match the traces of the obstructing newforms (for all primes of characteristic $<1000$ say). We also note that in both cases, the coefficient field of the obstructing newform is $\Q(\sqrt{-2})$. This is the same as the $\Q$-endomorphism algebra of $B = \mathrm{Res}_\Q^M(E)$, as expected.

When $q = 41$ or $q = 97$, the solutions in Theorem \ref{mainthm} with exponent $n = 7$ prevent us from eliminating the isomorphism of mod $n$ representations of $G$ and $F_q$ (as noted in Section 2), but these solutions do not pose any issue when working with the $\Q$-curve $E$. This is because the exponent $n=7$ is not large enough to force multiplicative reduction at the primes of $M$ above the rational prime $2$. A similar remark applies for (restricting attention to primes $q < 1000$)
$$
q \in \{ 233, 313, 401, 601 \},
$$
which are potentially accessible to the methods of this paper (though the corresponding computation of forms at level $2q^2$ would, with current techniques,  be formidable).


\bibliographystyle{plainnat}

\end{document}